\documentclass[onefignum,onetabnum]{siamart}

\usepackage[top=1in, bottom=1.25in, left=1in, right=1 in]{geometry}
\usepackage{amssymb, amsmath}
\usepackage{bbm} 
\usepackage{xcolor}
\usepackage{graphicx}
\usepackage{adjustbox}

\newcommand{\E}{\mathbb{E}}
\renewcommand{\Pr}{\mathbb{P}}

\newcommand{\bR}{\mathbb{R}}

\newcommand{\Ito}{It\^o }

\usepackage{lipsum}
\usepackage{amsfonts}
\usepackage{graphicx}
\usepackage{epstopdf}
\usepackage{algorithmic}
\ifpdf
  \DeclareGraphicsExtensions{.eps,.pdf,.png,.jpg}
\else
  \DeclareGraphicsExtensions{.eps}
\fi

\newtheorem{remark}{Remark}

\title{The optimal stopping time for a class of It\^o diffusion bridges\thanks{Submitted to the editors on March, 26th, 2019.
\funding{This research was partially supported by the Spanish Ministry of Economy and Competitiveness Grants MTM2017-85618-P via
FEDER funds and MTM2015-72907-EXP; both authors also thank the NYUAD, Abu Dhabi, United Arab Emirates, for hosting them during the fall  2018.}}}

\author{Bernardo D'Auria\thanks{UC3M, Department of Statistics, 28911, Legan\'{e}s, Spain \& UC3M-BS, Institute of Financial Big Data, 28903 Getafe, Spain 
  (\email{bernardo.dauria@uc3m.es}).}
\and Alessandro Ferriero\thanks{UAM \& ICMAT, Department of Mathematics, 28049 Madrid, Spain
  (\email{alessandro.ferriero@uam.es}).}}

\usepackage{amsopn}

\ifpdf
\hypersetup{
  pdftitle={Optimal stopping times for a class of It\^o diffusion bridges},
  pdfauthor={B. D'Auria and A. Ferriero}
}
\fi

\begin{document}

\maketitle

\begin{abstract}
The scope of this paper is to study the optimal stopping problems associated to a stochastic process, which may represent the gain of an investment, for which information on the final value is available a priori. 
This information may proceed, for example, from insider trading or from pinning at expiration of stock options. 

We solve and provide explicit solutions to these optimization problems.
As special case, we discuss different processes whose optimal barrier has the same shape as the optimal barrier of the Brownian bridge. So doing we provide a catalogue of alternatives to the Brownian bridge which in practice could be better adapted to the data. Moreover, we investigate if, for any given (decreasing) curve, there exists a process with this curve as optimal barrier. This provides a model for the optimal liquidation time, i.e. the optimal time at which the investor should liquidate a position in order to maximize the gain.
\end{abstract}



\section{Introduction}
The scope of this paper is to study the optimal stopping problems associated to a stochastic process $\{X_s\}$, for time $s$ in $[t,1]$, which may represent the gain of an investment, for which information on the final value, say at time $s=1$, is available a priori. 
This information may proceed, for example, from insider trading or from pinning at expiration time of stock options. 

Roughly speaking, the class of stochastic processes subject of our study is defined by bringing the infinite horizon mean-reverting Ornstein-Uhlenbeck process with constant parameters to a finite horizon.

In this work we solve and provide explicit solutions to the optimal stopping time problems associated to this class, which contains as particular case but it is not limited to the optimal stopping time associated with the Brownian bridge
	$$dX_s=-\frac{X_s}{1-s}ds+dB_s,\;\;\;s\in[t,1],\;\;\;\mbox{ with }X_t=x.$$
In this case it is known (see \cite{S69}, \cite{EW09}) that the optimal stopping time of
	$$\sup\{\E[X_\tau|X_t=x]:\tau\mbox{ stopping time in }[t,1]\}=:V(x,t)$$
is the one given by 
	$$\tau_*:=\inf\{s\in[t,1]:X_s\geq \beta_*\sqrt{1-s}\},$$
for an appropriate constant $\beta_*\approx0.84$. In other words, if the stock price is equal to the optimal barrier $\beta_*\sqrt{1-t}$ at time $t$, then $X_{t}= V(X_t,t)$ and thus the stock price is the maximum in expectation.

We also discuss different processes $\{X_s\}$ whose optimal barrier has the same shape as the optimal barrier of the Brownian bridge. So doing we provide a catalogue of alternatives to the Brownian bridge which in practice could be better adapted to the data. 

Moreover, we investigate if, for any given (decreasing) curve, there exists a process with this curve as optimal barrier. 
This provides a model for the optimal liquidation time, i.e. the optimal time at which the investor should liquidate a position in order to maximize her gain. 
More precisely, an investor takes short/long positions in the financial market based on her view about the future economy; 
for example, the real estate price is believed to increase in the long term, the inflation is believed to increase in the short term, the value of a certain company is believed to reduce in the medium term, etc. 
If $X_s$ is the market price at a time $s$ of a product of the market which is the object of the investor's view, the investor's view can be modeled as a deterministic function $\gamma(s)$ which represents the ``right'' expected value at time $s$ assigned by the investor at the time $t<s$, when the position needs to be taken, to the product. 
One may think of $\gamma(\cdot)$ as the short/medium/long period view. 
Of course, if the evolution of the economy induces the investor to distrust her initial view, the position is liquidated. 
Otherwise, if the investor maintains her view over time, we are interested in answering the following question:
what is the optimal time at which the investor should liquidate her position in order to maximize the gain? 

Similar problems have been studied in the recent literature. 
As mentioned at the beginning, the optimal stopping time associated to the case in which $\{X_s\}$ evolves like a Brownian bridge (finite horizon, i.e. $s\in[0,1]$) has been originally investigated in \cite{S69}, later with a different approach in \cite{EW09}, and the optimal barrier is equal to $\beta_*\sqrt{1-s}$. 
In \cite{BCSY15}, the authors study the double optimal stopping time problem, i.e. a pair of stopping times $\tau_1<\tau_2$ such that the expected spread between the payoffs at $\tau_1$ and $\tau_2$ is maximized, still associated to the case in which $\{X_s\}$ evolves like a Brownian bridge. 
The two optimal barriers which define the two optimal stopping times are found to have the same shape as in \cite{EW09}, i.e. $\beta_1\sqrt{1-s}$, $\beta_1\approx -0.56$, and $\beta_2\sqrt{1-s}$, $\beta_2=\beta_*$.
In \cite{ELT11}, the authors study the optimal stopping problem when $\{X_s\}$ evolves like a mean-reverting Ornstein-Uhlenbeck process (infinite horizon, i.e. $s\in[0,\infty)$), but in the presence of a stop-loss barrier, i.e. a level $B<0$ such that the position is liquidated as $X_s\leq B$. 
Furthermore, they study the dependency of the optimal barrier (which is a level $b>0$ in this case) to the parameters of the Ornstein-Uhlenbeck process and to the stop-loss level $B$. 

The paper is organized as follows. In Section \ref{sec:form}, we introduce the class of processes studied in our work. In Section \ref{sec:HJB} and \ref{sec:sol}, we find  the optimal barrier by using the Hamilton-Jacobi-Bellman equations and prove the optimality. In Section \ref{sec:conc}, we discuss the results contained in the paper. 





\section{The Formulation of the Problem}\label{sec:form}

Our problem can be formulated as follows. Let us consider a mean reverting Ornstein-Uhlenbeck process with constant coefficients $\theta\geq0$, $\sigma>0$:
	$$dZ_t = -\theta Z_t dt+\sigma dB_t, \;\;\; t>0,$$
where $\{B_t\}$ is the standard Brownian motion.
The mean of $Z_t$ converges to $0$ and its variance converges to $\sigma^2/(2\theta)$, as $t$ goes to $\infty$, which is $\infty$, if $\theta=0$. 

We want to adapt this model to a finite time horizon with final time, say, $t=1$. Moreover, as we want also to incorporate our view of the future, i.e. the function $\gamma$, in the model, if $b(s):=\gamma(s)-\gamma(1)$, $s\in[0,1]$, is a non-negative decreasing function, we map $s$ in $[0,1]$ to $t$ in $[0,\infty]$ by $t=-\ln [b(s)/b(0)]$ and we re-write the Ornstein-Uhlenbeck process as
 	$$d\left[\frac{X_s-\gamma(1)}{b(s)}\right]= -\theta\frac{X_s-\gamma(1)}{b(s)}d[-\ln b(s)]+\sigma d[B_{-\ln b(s)}], \;\;\; s\in(0,1],$$
where $X_s:=b(s)Z_{-\ln [b(s)/b(0)]}+\gamma(1)$. 

Note that $X_1=\gamma(1)$, as proved in Lemma \ref{lem:bridge_proof}, and the quantity 
$Z_{-\ln [b(s)/b(0)]}=[X_s-\gamma(1)]/b(s)$ thus represents the deviation of the market price to the final value in proportion to the deviation of the ``right'' value to the final value,
at time $s$. Note also that, despite there are different way to map $[0,1]$ to $[0,\infty]$ making use of the function $b$, the chosen map is the most natural. Indeed, with this choice we have the following equation for the expected value of $X_s-\gamma(1)$:
	$$d[\ln \E[X_s-\gamma(1)]]=(\theta+1) d[\ln b(s)],$$ 
i.e. the rate of change in logarithmic scale of $\E[X_s-\gamma(1)]$ is proportional to the same rate of $b(s)$.

For convenience, we relabel the parameters by $\theta:=\alpha^2$ and $\sigma:=\sqrt{2/\beta^2}$. 
Writing the equation for $\{X_s\}$,
starting at any time $t$ in $(0,1)$, we then obtain
	\begin{equation}\label{eq:def_X}
	dX_s = (1+\alpha^2)(X_s-\gamma(1))\frac{b'(s)}{b(s)}ds+\sqrt{\frac{-2b'(s)b(s)}{\beta^2}}dB_s, \;\;\;s\in(t,1),\;\;\; X_t=x.
	\end{equation}
	
We are interested in the optimization problem 
	\begin{equation}\label{eq:def_V}
	V(x,t):=\sup_{\tau}\E[X_\tau|X_t=x],
	\end{equation}
where $\tau\in[t,1]$ is a stopping time.
Here we assume that at the final time, which without loss of generality is normalized to $s=1$, the market price coincides with the ``right'' value, i.e. $X_1=\gamma(1)$, as proved in Lemma \ref{lem:bridge_proof}.

\begin{remark}
If $\alpha=1$, $\beta=\beta_*$, $\gamma(s)=\beta_*\sqrt{1-s}$, equation \eqref{eq:def_X} becomes:
	$$dX_s = -\frac{X_s}{1-s}ds+dB_s,$$
which is nothing but the Brownian bridge treated in \cite{EW09}, and derived as a particular case in our setting. 
\end{remark}

\begin{lemma}\label{lem:bridge_proof}
The process $\{X_s\}$ defined by \eqref{eq:def_X} is equal to
  \begin{equation}
    X_s=b^{1+\alpha^2}(s)\left[\frac{x-\gamma(1)}{b^{1+\alpha^2}(t)}+
  \int_t^s \sqrt{\frac{-2b'(r)}{\beta^2b^{1+2\alpha^2}(r)}} dB_r \right]+\gamma(1),\;\;\;s\in[t,1] 
  \end{equation}
Hence, $X_1=\gamma(1)$ and for $s$ in $[t,1]$,
  \begin{align}\label{eq:EandVar}
	\E[X_s] &= \gamma(1) + (x-\gamma(1)) \left[b(s)/b(t)\right]^{1+\alpha^2}\\
	\mbox{Var}[X_s] &= (b(s)/\beta)^2\cdot\left\{
	\begin{array}{ll}
    \displaystyle \frac{1-\left[b(s)/b(t)\right]^{2\alpha^2}}{\alpha^2}\ , &\alpha\neq0\\
     \displaystyle -\ln\left[b(s)/b(t)\right]^2\ ,
    &\alpha=0
  \end{array}\right.
	\end{align}
\end{lemma}
\begin{proof}
By multiplying the two terms of the equation \eqref{eq:def_X}  by $[b(t)/b(s)]^{1+\alpha^2}$, which is not identically zero, this can be rewritten as
	$$\left[\frac{b(t)}{b(s)}\right]^{1+\alpha^2}[dX_s - (1+\alpha^2)(X_s-\gamma(1))\frac{b'(s)}{b(s)}ds]=
	\left[\frac{b(t)}{b(s)}\right]^{1+\alpha^2}\sqrt{\frac{-2b'(s)b(s)}{\beta^2}}dB_s.$$
Since the \Ito derivative of $(X_s-\gamma(1))[b(t)/b(s)]^{1+\alpha^2}$ is equal to
	$$d\left[(X_s-\gamma(1))\left[\frac{b(t)}{b(s)}\right]^{1+\alpha^2}\right]
	=\left[\frac{b(t)}{b(s)}\right]^{1+\alpha^2}[dX_s - (1+\alpha^2)(X_s-\gamma(1))\frac{b'(s)}{b(s)}ds],$$
we have that
	$$(X_s-\gamma(1))\left[\frac{b(t)}{b(s)}\right]^{1+\alpha^2}=X_t-\gamma(1)+
	\int_t^s\left[\frac{b(t)}{b(r)}\right]^{1+\alpha^2}\sqrt{\frac{-2b'(r)b(r)}{\beta^2}}dB_r.$$
This implies the expression for  $X_s$ given in the statement.
Moreover, the formulas for the mean and the variance of $X_s$ can be directly derived from this expression and the \Ito isometry. 

Lastly, as the $\mbox{Var}[X_s]\to0$, when $s$ converges to $1$, and since $\{X_s\}$ is continuous, we obtain $X_1=\gamma(1)$. This completes the proof.
\end{proof}

\begin{remark}
	The statements of the Lemma \ref{lem:bridge_proof} on the mean and variance of $\{X_s\}$ are consistent with the corresponding quantities for the Ornstein-Uhlenbeck process $\{Z_t\}$ introduced at the beginning of the section, i.e. $\E[(X_s-\gamma(1))/b(s)]\to 0$ and $\mbox{Var}[(X_s-\gamma(1))/b(s)]\to 1/(\alpha\beta)^2=\sigma^2/(2\theta)$, as $s\to1$.
	
	Moreover, observe that $\mbox{Var}[X_s]$ depends continuously on $\alpha$ also in $\alpha=0$ since $(1-\left[b(s)/b(t)\right]^{2\alpha^2})/\alpha^2$ converges to $-\ln\left[b(s)/b(t)\right]^2$, as $\alpha\to0$.
\end{remark}

\section{The HJB Equation}\label{sec:HJB}

The Hamilton-Jacobi-Bellman (HJB) equation for the function $V$ defined in \eqref{eq:def_V} are given by (see \cite{PS06}): 
\begin{equation}\label{eq:for_V}
\begin{array}{ll}
\left\{\begin{array}{ll}
\displaystyle V_t(x,t) + (1+\alpha^2)(x-\gamma(1)) \frac{b'(t)}{b(t)} V_{x}(x,t) -\frac{b'(t)b(t)}{\beta^2}V_{xx}(x,t) =0,& x < a(t) + \gamma(1)\\
V(x,t) = x, & x \geq a(t) + \gamma(1) \\
V_x(x,t)=1, & x = a(t) + \gamma(1) \\
V(x,t)\to\gamma(1),& x\to-\infty
\end{array}\right.
& , t\in[0,1]
\end{array}
\end{equation}
where $a:[t,1]\to\bR$ is the free boundary such that the stopping time
	$$\tau^*:=\inf\{s\in[t,1):X_s\geq a(s) + \gamma(1)\}$$
is the optimal stopping time for the problem \eqref{eq:def_V}, i.e. $V(x,t)=\E[X_{\tau^*}]$.

The three boundary conditions in \eqref{eq:for_V} are necessary, but not sufficient, conditions which $V$ must satisfy in order to be a candidate solution of the optimal stopping problem. 
Indeed, the first and second conditions are nothing but that $X_s+\gamma(1)$ is equal to the maximum in expectation when we stop the process, i.e. $V(a(s)+\gamma(1),s)=a(s)+\gamma(1)$, and that if $x\geq a(s)+\gamma(1)$, then $V(x,s)=x$ and thus, as $V$ is smooth, $V_x(a(s)+\gamma(1),s)=1$. The last condition is due to the fact that the more negative the process starts off, the more difficult the process overcomes $\gamma(1)$.

If we assume that, given $y\in\bR$, $V(y a(t) + \gamma(1),t)-\gamma(1)$ has a value proportional to $a(t)$ independently on $t$, and $a(t)=b(t)$, i.e. 
$$ V(x,t) - \gamma(1) = f\left(\frac{x-\gamma(1)}{b(t)}\right) b(t), $$
then \eqref{eq:for_V} can be rewritten as
\begin{equation}\label{eq:for_f}
\left\{\begin{array}{ll}
\displaystyle f(y)+\alpha^2y f'(y)-\frac{1}{\beta^2}f''(y)=0,& y < 1\\
f(1)=1&\\
f'(1)=1&\\
f(y)\to0,& y\to-\infty
\end{array}\right.
\end{equation}
where $y =(x-\gamma(1))/b(t)$.

\begin{lemma}\label{lem:sol_ODE}
The equation \eqref{eq:for_f} admits a unique solution, for any $\alpha\geq0$, and for a unique $\beta=\beta(\alpha)$, which depends on $\alpha$.
\end{lemma}
\begin{proof}
If $\alpha=0$, then the equation in $\eqref{eq:for_f}$ becomes
	$$f(y)-\frac{1}{\beta^2}f''(y)=0,$$
which has general solution $f(y)=c_1e^{\beta y}+c_2e^{-\beta y}$, with $c_1,c_2\in\bR$. The boundary conditions in \eqref{eq:for_f} give constraints on the parameters. Indeed,
we have that
\begin{align*}
  \lim_{y \to -\infty} f(y)= 0 
    &\implies c_2 = 0 \\
  f(1) = 1 
  & \implies  c_1= e^{-\beta}.
\end{align*}
Therefore the solution assumes the form $f(y) = e^{\beta (y-1)}$
Finally, the last boundary condition $f'(1) = 1$ implies that $\beta=1$ and the solution is
	$$f(y)=e^{y-1}.$$
Therefore, when $\alpha=0$ the equation with boundary conditions $\eqref{eq:for_f}$ has solution if and only if $\beta=\beta(0):=1$.

Assume now $\alpha>0$.
By substituting $f(y) =: h(-\beta y)$, $\beta>0$, and using $x= -\beta y$, we can rewrite the differential equation in \eqref{eq:for_f} without boundary conditions by
  \begin{equation}\label{h.diff.eq}
    h(x) + \alpha^2 x h'(x) = h''(x),\;\;\;x>-\beta \ .
  \end{equation}   
This can be rewritten as
	$$\frac{d}{dx} \left[ e^{-(\alpha x)^2/2} h'(x) \right]=  e^{-(\alpha x)^2/2} h(x).$$
Finally, with a further substitution $h(x) =: u(\alpha x) \exp[(\alpha x)^2/4]$
and $z= \alpha x$, we obtain the so called parabolic cylinder differential equation
$$  
  u''(z) = \left(\frac{z^2}{4} + 
    \frac{1}{\alpha^2} - \frac{1}{2} 
    \right) u(z),\;\;\;z>-\alpha\beta.$$
Two linear independent solutions of the parabolic cylinder differential equation are
\begin{align*}
  \hat u_1( z) &=  \exp\left[-\frac{1}{4} z^2 \right]
    M\left(\xi, \frac{1}{2} , \frac{1}{2} z^2\right)\\
  \hat u_2( z) &= z \exp\left[-\frac{1}{4} z^2 \right]
    M\left(\xi + \frac{1}{2} , \frac{1}{2} + 1 , \frac{1}{2} z^2\right)
\end{align*}
where $\xi := 1/(2\alpha^2) $
and $M(a,b,z)$ is the Kummer's function (see \cite{AS64}  Chapter 13). The functions $\hat u_1$ and $\hat u_2$ are an even and a odd function, respectively, and depend on the parameter $\xi$. However, in order to keep the notation as simple as possible, the dependence on the parameter is not explicitly marked.

For convenience, let us introduce other two independent solutions $ u_1$ and $ u_2$ defined 
by the following linear combinations of $\hat u_1$ and $\hat u_2$:

\begin{align*}
  u_1( z) 
  &=  \frac{\sqrt{\pi}}{2^\xi} \left[
     \frac{1}{\Gamma \left(1/2+\xi\right)} \hat u_1(z)  
     - \frac{\sqrt{2}}{\Gamma (\xi )}  \hat u_2(z)
     \right]\\ 
  u_2(z) 
  &= 2^\xi \left[
        \frac{ \sin(\xi\pi) }{\Gamma (1-\xi )} \hat u_1(z)
      +\sqrt{2} \frac{ \cos(\xi\pi) }{\Gamma \left(1/2-\xi \right)} \hat u_2(z)
      \right]
\end{align*}
where $\Gamma$ is the gamma function.
The expressions above are well defined for all $\xi\in\bR$. 

Note that the function $ u_1$  converges to $0$ and $u_2$ diverges, as $|z|$ goes to $\infty$.
Moreover,
\begin{equation} \label{U.infty.behavior}
\begin{array}{l}
  \displaystyle\lim_{|z|\to\infty}  e^{ z^2/4} z^{2\xi}   u_1(z) = 1 \\
\displaystyle  \lim_{|z|\to\infty} e^{-z^2/4} z^{1-2\xi} u_2(z) = \sqrt{\frac{\pi}{2}} 
\end{array}
\end{equation}
Hence, using the transformation defined between $u$ and $h$, we have that   
\begin{align}
  h_1(z) 
  &:=  u_1(z)  \exp[z^2/4] \notag \\ 
  &=  \frac{\sqrt{\pi}}{2^\xi} \left[
  \frac{1}{\Gamma\left(1/2+\xi\right)}
    M\left(\xi, \frac{1}{2} , \frac{1}{2} z^2\right)
  -\frac{z\sqrt{2}}{\Gamma\left(\xi\right)}  
    M\left(\xi + \frac{1}{2} , \frac{1}{2} + 1 , \frac{1}{2} z^2\right)
  \right] \label{h1.def}\\ 
  h_2( z) 
  &:=  u_2(z) \exp[z^2/4] \notag \\
  &= 2^\xi \left[
    \frac{ \sin(\xi\pi) }{\Gamma (1-\xi )} 
    M\left(\xi , \frac{1}{2} , \frac{1}{2} z^2\right)
  +\frac{z \sqrt{2} \cos(\xi\pi) }{\Gamma \left(1/2-\xi \right)} 
    M\left(\xi + \frac{1}{2} , \frac{1}{2} + 1 , \frac{1}{2} z^2\right)
  \right] 
\end{align}
are, by \eqref{U.infty.behavior},  such that
\begin{align*}
 h_1( z) &\sim \frac{1}{ z^{2\xi}}  \\
 h_2( z) &\sim \sqrt{\frac{\pi}{2}} e^{z^2/2} z^{2\xi-1} 
\end{align*}
for $|z|\to\infty$. 

Therefore, going back to the the differential equations in \eqref{eq:for_f} without boundary conditions, we can write all its solutions by
\begin{align}
  f(y) = c_1 h_1( -\alpha \beta y)  
  + c_2 h_2( -\alpha \beta y),
\end{align}
with $c_1,c_2\in\bR$,
where the parameter $\xi$ in the definition of $h_1$ and $h_2$ is equal to $1/(2\alpha^2)$.

The boundary conditions in \eqref{eq:for_f} give constraints on the parameters. Indeed,
we have that
\begin{align*}
  \lim_{y \to -\infty} f(y)= 0 
    &\implies c_2 = 0 \\
  f(1) = 1 
  & \implies  c_1= \frac{1}{h_1(-\alpha\beta)}.
\end{align*}
Therefore the solution assumes the form
\begin{align}\label{eq:formula_h}
  f(y) 
  &= \frac{ h_1( -\alpha\beta y)}{ h_1(-\alpha\beta)}.
\end{align}
Finally, the last boundary condition $f'(1) = 1$ implies
\begin{equation}\label{const.eq}
     -\alpha\beta\frac{ h_1'( -\alpha\beta)}{ h_1(-\alpha\beta)}=1 \ .
\end{equation}
For every given $\alpha$, by Lemma \ref{lm.uniq.sol} it exists a unique negative solution $x_\alpha$ to the equation $x_\alpha h_1'( x_\alpha)=h_1(x_\alpha)$. 
Therefore, if $\beta(\alpha):=-x_\alpha/\alpha$ (see Figure \ref{fig:parameters}), 
then the differential equation with boundary conditions \eqref{eq:for_f} admits the unique solution given by the formula \eqref{eq:formula_h} with $\beta=\beta(\alpha)$. 
This completes the proof.
\begin{figure}
\centering
  \includegraphics[height=5cm]{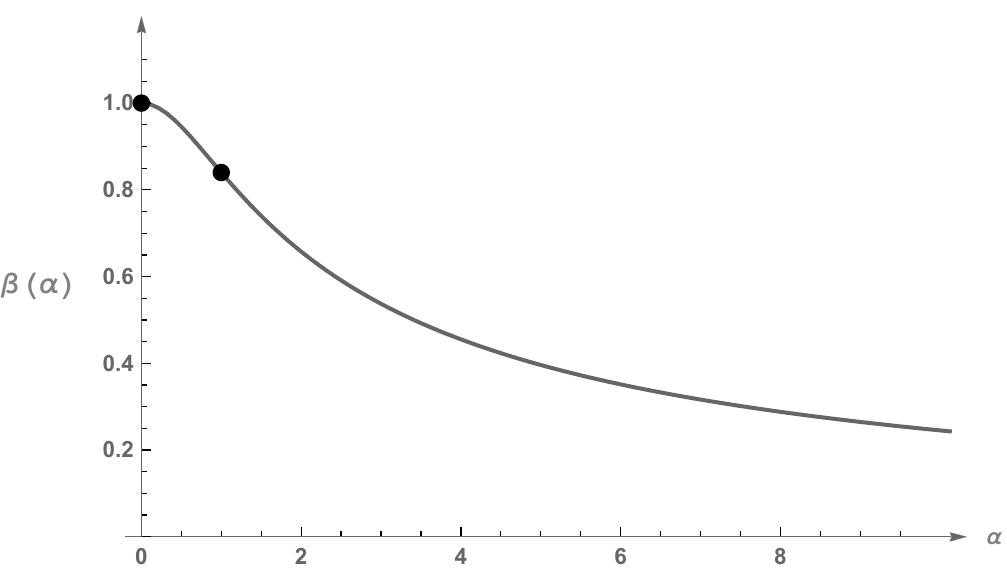}
  \caption{\label{fig:parameters} Plotting $\beta$ as a function of $\alpha$. The marked points are $\beta(0)=1$ and $\beta(1)\approx 0.839924$. }
\end{figure}
\end{proof}

\begin{remark}
Considering again the case of the Brownian bridge treated in \cite{EW09}, $\alpha=1$, implies that
\begin{align*}
h_1(z) 
&= \frac{1 - \Phi(z)}{\phi(z)}
\end{align*}
where $\Phi(z) = \Pr[Z \leq z]$ with $Z$ a standard Normal distributed and $\phi(z)=\Phi'(z)$, and the equation \eqref{const.eq} becomes
$$
\frac{\phi(\beta(1))}{\Phi(\beta(1))}
= \frac{1-\beta^2(1)}{\beta(1)}
$$
that gives as unique solution 
$\beta(1) \approx 0.84$
as already found in \cite{S69} and \cite{EW09}.
\end{remark}

\section{The Solution}\label{sec:sol}

In this section we verify that the solution $V$ to the HJB equation found in Section \ref{sec:HJB} is indeed solution to the optimal stopping problem \eqref{eq:def_V}. This is a needed verification as we remind that the HJB equation represents only a necessary condition for the solution to the optimal stopping time problem, but it is not a sufficient condition.

\begin{theorem}\label{thm:proof_opt}
The function 
	$$V^*(x,t):=\gamma(1)+\left\{
	\begin{array}{ll}
	\displaystyle h_1\left(-\alpha\beta(\alpha)\frac{x-\gamma(1)}{b(t)}\right)
	\frac{b(t)}{h_1(-\alpha\beta(\alpha))},&x<b(t)+\gamma(1)=\gamma(t)\\
	x,&x\geq \gamma(t)
	\end{array}\right.$$
is equal to $V(p,t)$ defined in \eqref{eq:def_V} and 
	$$\tau^*=\inf\{s\in[t,1]:X_s\geq \gamma(t)\},$$ 
i.e. $a(t)=b(t)$, $t\in[0,1]$, is the optimal stopping time.
\end{theorem}
\begin{proof}
Consider the process $\{V^*(X_s,s)\}$. From the \Ito formula and the definition of $V^*$, it follows that
$$
\begin{array}{l}
\displaystyle dV^*(X_s,s)=V_t^*(X_s,s)ds+V_x^*(X_s,s)dX_t-\frac{b'(s)b(s)}{\beta^2}V_{xx}^*(X_s,s)ds\\
 \displaystyle\hspace{4.8em}=(1+\alpha^2)(X_s-\gamma(1))\frac{b'(s)}{b(s)}{\mathbbm 1}_{[b(s),\infty)}(X_s)ds+V_x^*(X_s,s)\sqrt{\frac{-2b'(s)b(s)}{\beta^2}}dB_s. 
\end{array}
$$
Note that $\{\int_t^sV_x^*(X_r,r)\sqrt{-2b'(r)b(r)/\beta^2}dB_r\}$ is a martingale as 
	$$V_x^*(X_r,r)\sqrt{-2b'(r)b(r)/\beta^2}=
	V_x^*(X_r,r)\sqrt{-2b'(r)b(r)/\beta^2}{\mathbbm 1}_{(-\infty,b(r))}(X_r)+
	\sqrt{-2b'(r)b(r)/\beta^2}{\mathbbm 1}_{[b(r),\infty)}(X_r)
	$$
is bounded.

Let $\tau\in[t,1]$ be a stopping time. Since $V^*(x,t)\geq x$ as proved in Lemma \ref{lem:sign_f}, since the variable $(1+\alpha^2)(X_s-\gamma(1))[b'(s)/b(s)]{\mathbbm 1}_{[b(s),\infty)}(X_s)$ is non-positive and since $\{\int_t^sV_x^*(X_r,r)\sqrt{-2b'(r)b(r)/\beta^2}dB_r\}$ is a martingale, then, by the Optional Sampling theorem,
$$\begin{array}{l}
\displaystyle \E[X_{\tau}]\leq\E[V^*(X_\tau,\tau)]\\
\hspace{3.8em}\displaystyle=V^*(x,t)+\E\left[\int_t^\tau(1+\alpha^2)(X_r-\gamma(1))\frac{b'(r)}{b(r)}{\mathbbm 1}_{[b(r),\infty)}(X_r)dr\right]\\
\displaystyle\hspace{4.8em}+\E\left[\int_t^\tau V_x^*(X_r,r)\sqrt{\frac{-2b'(r)b(r)}{\beta^2}}dB_s\right]\\
\hspace{3.8em}\displaystyle=V^*(x,t)+\E\left[\int_t^\tau(1+\alpha^2)(X_r-\gamma(1))\frac{b'(r)}{b(r)}{\mathbbm 1}_{[b(r),\infty)}(X_r)dr\right]\leq V^*(x,t).
\end{array}$$
Hence, since this is true for any stopping time $\tau$, we have that $V(x,t)\leq V^*(x,t).$

On the other hand, since $V^*(X_{\tau^*},\tau^*)=X_{\tau^*}$, then $(1+\alpha^2)(X_s-\gamma(1))[b'(s)/b(s)]{\mathbbm 1}_{[b(s),\infty)}(X_s)=0$, for every $s$ in $[t,\tau^*)$, and thus
$$V(x,t) \geq\E[X_{\tau^*}]=\E[V^*(X_{\tau^*},\tau^*)]=V^*(x,t)+\E\left[\int_t^{\tau^*}V_x^*(X_r,r)\sqrt{\frac{-2b'(r)b(r)}{\beta^2}}dB_s\right]
=V^*(x,t).
$$
This completes the proof.
\end{proof}

\begin{lemma}\label{lem:sign_f}
Any solution $f\in{\mathbb C}^2((-\infty,1])$ to \eqref{eq:for_f}, with $\alpha\geq0$, is such that $f(y)\geq \max\{0,y\}$, for all $y\in(-\infty,1]$.
\end{lemma}
\begin{proof}
Here we show the result without making use of Lemma \ref{lem:sol_ODE}.

If $\alpha=0$, then by the equation in \eqref{eq:for_f}, $f''(y)=\beta^2 f(y)$ and thus $f$ is convex where $f$ is positive and concave where $f$ is negative. As $f(1)=1$, then $f$ is convex in $1$. Thus, $f$ cannot touch the straight line $l(y)=y$ in a point $\bar y\in[0,1)$, as otherwise there would be a point $\tilde y\in(\bar y,1)$ such that $f''(\tilde y)=0$, which is impossible as $f''(\tilde y)= \beta^2 f(\tilde y)> 0$. Hence $f(y)>y$, for all $y\in[0,1)$.

Similarly, if $f(\bar y)$ is negative, for a $\bar y<0$, then $f$ is concave in $\bar y$. As $f(y)\to0$, for $y\to-\infty$, then there would be $\tilde y<\bar y$ such that $f''(\tilde y)=0$, which is impossible as $f''(\tilde y)= \beta^2 f(\tilde y)< 0$. Hence, $f(y)>0$, for all $y\in(-\infty,0]$.

If $\alpha>0$, then by the equation in \eqref{eq:for_f}, $f''(y)=\beta^2 [f(y)+\alpha^2yf'(y)]$. If $f$ touches the straight line $l(y)=y$ in a point $\bar y\in[0,1)$, then it would exist a point $\tilde y\in(\bar y,1)$ such that $f''(\tilde y)=0$. Hence, $f(\tilde y)+\alpha^2yf'(\tilde y)=0$ and so $f'(\tilde y)<0$. Therefore, it would exists a local maximum $\hat y\in(\bar y,\tilde y)$, i.e. $f'(\hat y)=0$, $f''(\hat y)\leq0$, which is impossible as $f''(\hat y)= \beta^2 f(\hat y)> 0$. Hence $f(y)>y$, for all $y\in[0,1)$.

Similarly,  if $f(\bar y)$ is negative, as $f(y)\to0$, for $y\to-\infty$, there exists a local minimum $\hat y<\bar y$, i.e. $f'(\hat y)=0$, $f''(\hat y)\geq 0$, which is impossible as  $f''(\hat y)= \beta^2 f(\hat y)<0$. Hence, $f(y)>0$, for all $y\in(-\infty,0]$.
This completes the proof.
\end{proof}

\section{Conclusion}\label{sec:conc}
Theorem \ref{thm:proof_opt} gives the explicit solution to the optimization problem \eqref{eq:def_V}. 

Moreover, suppose that $\gamma(t):=\beta(\alpha)\sqrt{1-t}$. Thus, $b(t)=\gamma(t)$ since $\gamma(1)=0$. Then Theorem \ref{thm:proof_opt} states that the optimal stopping time associated to the process
	$$dX_s=-\frac{(1+\alpha^2)}{2}\frac{X_s}{1-s}ds+dB_s, \;\;\;s\in[t,1]$$
with $X_t=x$,
is 
    $$\tau^*=\inf\{s\in[t,1]:X_s\geq\gamma(t) \}$$ 
and 
	$$V(x,t)=\sup_\tau \E[X_\tau]= \E[X_{\tau^*}]=\left\{
	\begin{array}{ll}
	\displaystyle h_1\left(-\alpha\beta(\alpha)\frac{x}{\gamma(t)}\right)
	\frac{\gamma(t)}{h_1(-\alpha\beta(\alpha))},&x<\gamma(t)\\
	x,&x\geq \gamma(t)
	\end{array}\right.$$

As consequence of our result we obtain that if the drift in the \Ito representation of the Brownian bridge is multiplied by a constant factor $(1+\alpha^2)/2$, then the associated optimal barrier has the same shape as the barrier of the Brownian bridge multiplied by the constant factor $\beta(\alpha)/\beta(1)$. 

\begin{figure}
  \centering
  \includegraphics[clip, trim=0cm 1.2cm 0cm 1.2cm, width=.95\textwidth]{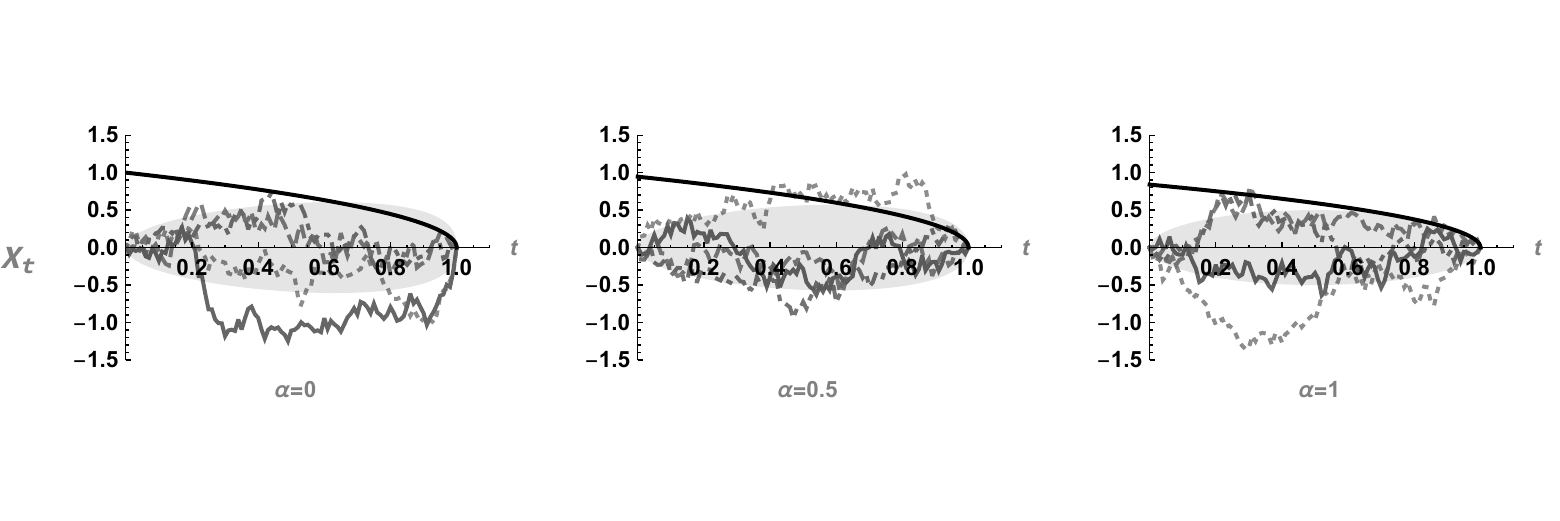}
  \includegraphics[clip, trim=0cm 1.2cm 0cm 1.2cm, width=.95\textwidth]{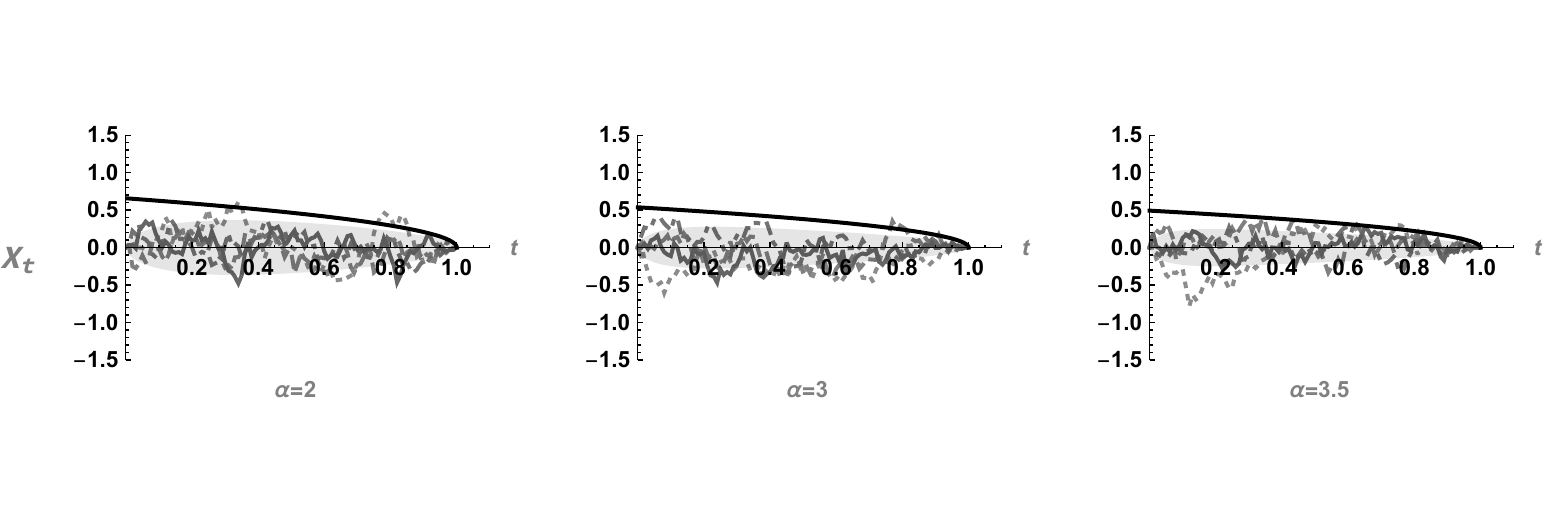}
    \caption{\label{fig:simulations} Four simulations of $X_s$, $s\in[0,1]$, as defined in \eqref{eq.spec.proc} with $t=0$, $x=0$, for different values of $\alpha$. 
    The light grey areas represent one standard deviation of $X_s$ above and below its expected value which is identically zero in the cases shown in the plots.
    The black solid lines show the optimal stopping boundaries.}
\end{figure}

Note that for $\alpha\geq0$, $\alpha\neq 1$, the process $\{X_s\}$ defined above is not a Brownian bridge. Indeed, by Lemma \ref{lem:bridge_proof},	
\begin{equation}\label{eq.spec.proc}
  X_s=x+\int_t^s\left[\frac{1-s}{1-r}\right]^{(1+\alpha^2)/2}dB_r \ , \quad  s\in [t,1] 
\end{equation}
This class of processes is such that the associated optimal barriers have the same shape as the optimal barrier of the Brownian bridge (see Figure \ref{fig:simulations}). Therefore, this class provides a catalogue of alternative bridges to the Brownian bridge which in practice could be better adapted to the data. 

Moreover, as consequence of Theorem \ref{thm:proof_opt}, we have proved that, for any given (decreasing) curve, there exist a process with this curve as optimal barrier. This provides a model for the optimal liquidation time, i.e. the optimal time at which the investor should liquidate a position in order to maximize the gain.

\appendix
\section{Technical lemma} 
\begin{lemma} \label{lm.uniq.sol} There exists a unique $x_\alpha<0$ such that  
  \begin{equation}\label{beta.cond}
    x_\alpha h_1'( x_\alpha)=h_1(x_\alpha) \ .
  \end{equation}
\end{lemma}
\begin{proof}
  By the definition of $h_1(x)$ in \eqref{h1.def}, $h_1(x)$ is positive for $x<0$ 
  and  $h_1(x)/x\to-\infty$ as $x\to-\infty$.
  
  The condition \eqref{beta.cond} implies that $x_\alpha\not=0$ and that the function $h_1(x)/x$ has a critical point at $x_\alpha$.
  Indeed taking derivatives we have
  $$\partial_x \left( \frac{h_1(x)}{x} \right) = \frac{1}{x^2} \left(x \, h'_1(x) - h_1(x)\right)$$
  and equating this expression to $0$ is the same as \eqref{beta.cond}.

  Now using \eqref{h.diff.eq} we have that
  \begin{equation}\label{sign.2nd.der}
    h''_1(x_\alpha) 
      = h_1(x_\alpha) + \alpha^2 x_\alpha h'_1(x_\alpha)
      = h_1(x_\alpha) (1 + \alpha^2) 
      > 0 \ .
  \end{equation}
  Computing the second derivative of $h_1(x)/x$ at $x_\alpha$ we have
  \begin{equation}
    \left.\partial^2_x \left( \frac{h_1(x)}{x} \right)\right|_{x_\alpha}
    = \left. \left(\frac{h_1''(x)}{x} - \frac{2}{x^3} \left(x \, h_1'(x)-h_1(x)\right) \right)\right|_{x_\alpha}
    = \frac{h_1''(x_\alpha)}{x_\alpha}
  \end{equation}
  and, using \eqref{sign.2nd.der}, it follows that $x_\alpha$ is a maximum point when $x_\alpha<0$ (and it would be a minimum point when $x_\alpha>0$).
  
  Now considering $x<0$, as $\beta(\alpha) = -x_\alpha/\alpha$,
  it follows that any critical point of $h_1(x)/x$ is a local maximum, and therefore there cannot be more than one.

  If there were two local maxima, this would imply a local minimum between them (if the function is not constant).
  Since $\lim_{x\to-\infty} h_1(x)/x = \lim_{x\to0^-} h_1(x)/x = -\infty$, there is a unique $x_\alpha < 0$ satisfying \eqref{beta.cond}. This completes the proof.
\end{proof}


\bibliographystyle{siamplain.bst}
\bibliography{references}

\end{document}